\documentclass{scrartcl}

\usepackage[english]{babel}
\usepackage[utf8]{inputenc}
\usepackage{mathtools}
\usepackage{mathpazo}
\usepackage{amsmath}
\usepackage{amssymb}
\usepackage{amsthm}
\usepackage{accents}
\usepackage{graphicx}
\usepackage[caption=false]{subfig}
\usepackage[top=1.25in]{geometry}
\usepackage{xspace}
\usepackage{ifthen}
\usepackage{enumitem}
\usepackage{todonotes}
\setenumerate{label=\textup{(\alph*)}}

\usepackage[hyphens]{url}
\usepackage[hidelinks]{hyperref}
\usepackage{cleveref}

\newcommand{\functionArgument}[1]{\ifthenelse{\equal{#1}{}}  
	{}
	{({#1})}
}

\newcommand{\wprox}[3][]{\mathrm{prox}_{#2}\functionArgument{#3}}
\newcommand{\prox}[2]{\wprox[]{#1}{#2}}
\newcommand{\wproj}[3][]{\mathrm{proj}_{#2}\functionArgument{#3}}
\newcommand{\proj}[2]{\wproj[]{#1}{#2}}

\newcommand{\domain}{D}
\newcommand{\gateaux}{G\^ateaux}
\newcommand{\lb}{\mathfrak{l}}
\newcommand{\ub}{\mathfrak{u}}

\newcommand{\Du}{\mathrm{D}}

\newtheorem{lemma}{Lemma}
\newtheorem{theorem}{Theorem}

\theoremstyle{definition}
\newtheorem{remark}{Remark}
\newtheorem{example}{Example}
\newtheorem{assumption}{Assumption}

\crefdefaultlabelformat{#2\textup{#1}#3}
\Crefname{assumption}{Assumption}{Assumptions}
\Crefrangeformat{assumption}{Assumptions #3\textup{#1}#4--#5\textup{#2}#6}

\newcommand{\keywords}[1]
{
	{\small	
		\textbf{Key words.} {#1}
		\\
	}
}

\newcommand{\amssubject}[1]
{
	{\small
		\textbf{AMS subject classifications.} {#1}
	}
}

\renewcommand{\abstract}[1]
{
	{\small
		\textbf{Abstract.} {#1}
		\\
	}
}

\title{Criticality measure-based error estimates for 
	infinite dimensional optimization}
\author{Danlin Li\thanks{H.\ Milton Stewart School of Industrial and Systems Engineering, Georgia Institute of Technology, Atlanta, Georgia 30332 (\texttt{dli620@gatech.edu},\texttt{johannes.milz@isye.gatech.edu}).}
	\and Johannes Milz\footnotemark[1]
}
\date{February 23, 2024}

\begin{document}
	
	\maketitle
	
	\abstract{%
		Motivated by optimization
		with differential equations, we consider optimization problems
		with Hilbert spaces as decision spaces. 
		As a consequence of their
		infinite dimensionality, the numerical solution necessitates
		finite dimensional approximations and discretizations.
		We develop an approximation framework and demonstrate
		criticality measure-based error estimates. 
		We consider criticality measures inspired by those used
		within optimization methods, such as semismooth Newton
		and (conditional) gradient methods.
		Furthermore, we show that our error estimates are order-optimal.
		Our findings augment existing distance-based error estimates,
		but do not rely on  strong convexity 
		or second-order sufficient
		optimality conditions. Moreover, 
		our error estimates can be  used  for code verification
		and validation.
		We illustrate our theoretical convergence rates on 
		linear, semilinear, and bilinear PDE-constrained optimization.
	}
	\par 
	\par 
	\amssubject{%
		90C15, 90C60,   35Q93, 35R60, 49M25,  49N10, 65M60, 65C05
	}

	\section{Introduction}
	
	Infinite dimensional optimization is prevalent across fields such as science, engineering, operations research, and applied mathematics, where decision variables are elements of  infinite dimensional spaces. 	
	Optimization problems governed by partial differential equations (PDEs)
	\cite{Manns2022,Milz2020a},
	optimal transport \cite{Lorenz2021},
	functional regression \cite{Nickl2020}, 
	statistical inverse problems \cite{Blanchard2018}, 
	and
	topology optimization \cite{Haubner2023}
	demonstrate a breath of contemporary 
	optimization problems with infinite dimensional decisions spaces.

	For their efficient numerical solution, infinite dimensional optimization 
	typically require discretizations of decision spaces and
	approximations of objective and constraint functions. These 
	discretization and approximation schemes  
	yield finite dimensional optimization problems, aimed at approximating
	their infinite dimensional counterparts.  Critical points 
	of the finite dimensional optimization problems provide estimators
	to those of their infinite dimensional counterparts. 
	Our main contribution is the derivation of order-optimal error estimates
	for a large class of optimization problems, and three types
	of criticality measures.
	The derivation of error estimates for discretizations
	of infinite dimensional optimization problems is a classical topic. 
	While we consider our approach to quantify 
	discretization accuracy canonical, it is novel and allows for
	order-optimal error estimates for a large problem class. 
	Besides rigorously quantifying errors, our error bounds
	can be used for code verification and validation 
	of large-scale optimization codes.
	Moreover, our error estimates augment existing
	error indicators, which  provide upper bound on the distance
	between (local) solutions to the finite dimensional and
	infinite dimensional problems.

	We consider the 
	composite optimization problem
	\begin{align}
		\label{eq:intro-infiniteproblem}
		\min_{u \in U}\, \widehat{J}(u) + \psi(u)
	\end{align}
	and its discretization given by
	\begin{align}
		\label{eq:intro-discretizedproblem}
		\min_{u_h \in U_h}\, \widehat{j}_h(u_h) + \psi_h(u_h),
	\end{align}
	where $h> 0$ is a discretization parameter, 
	$U$ is a real Hilbert space,  $U_h \subset U$ is a 
	finite dimensional subspace, 
	$\widehat{J}$ and $\widehat{j}_h$ are continuously differentiable
	with $\widehat{j}_h$ approximating $\widehat{J}$.
	Moreover, $\psi$ and $\psi_h$ are convex, proper,
	and  closed with $\psi_h$ approximating $\psi$.
	
	Let $\chi$ be a criticality measure for 
	\eqref{eq:intro-infiniteproblem}, and 
	$\chi_h$ be one for \eqref{eq:intro-discretizedproblem}. 
	The manuscript's main focus is on establishing 
	order-optimal error decompositions of the form:
	for all $u_h \in U_h$ and $h > 0$,
	\begin{align}
		\label{eq:intro-error-decomposition}
		\chi(u_h)
		\leq 
		\underbrace{\chi_h(u_h) }_{\text{optimization error}}
		+ \quad \text{discretization errors
			as a function of } h.
	\end{align}
	The decomposition \eqref{eq:intro-error-decomposition} 
	relates approximate criticality of the finite dimensional 
	problem to approximate criticality
	of the infinite dimensional discretization. 
	The ``optimization error'' is a quantity 
	to be  controlled 
	via termination tolerances of numerical optimization methods.
	The ``discretization error'' is an error contribution  depending
	on the approximation schemes used to formulate
	\eqref{eq:intro-discretizedproblem}
	and the infinite dimensional problem's properties.

	We derive error estimates for three different criticality measures
	and demonstrate optimality of our estimates.
	Let $\tau > 0$ be a parameter.
	We consider the 
	the normal map-based criticality measures
	\cite{Mannel2020,Milzarek2023,Robinson1992}
	\begin{align*}
		\chi_{\text{nor}}(v;\tau) & \coloneqq  
		\|\tau(v-\prox{\psi/\tau}{v}) + \nabla \widehat{J}(\prox{\psi/\tau}{v})\|_U,
	\end{align*}
	canonical criticality measures
	\cite{Milzarek2023,Ulbrich2011}
	\begin{align*}
		\chi_{\text{can}}(u; \tau)
		& \coloneqq \|u-\prox{\psi/\tau}{u-(1/\tau)\nabla \widehat{J}(u)}\|_U,
	\end{align*}
	and gap functions
	\cite{Hearn1982,Kunisch2022} (or optimality functions)
	\cite{Polak1997,Royset2012a})
	\begin{align*}
		\chi_{\text{gap}}(u)
		& \coloneqq   \sup_{v \in \mathrm{dom}(\psi)} \, \big\{\,
		(\nabla \widehat{J}(u), u-v)_{U} + \psi(u)- \psi(v)
		\,\big\}.
	\end{align*}
	The canonical and normal-map
	criticality measures are employed as termination 
	criteria in 
	gradient-based schemes \cite{Azmi2023,Milzarek2023}, 
	trust-region methods \cite{Baraldi2023,Ouyang2021}, 
	and semismooth Newton methods \cite{Mannel2020,Ulbrich2011}	
	while the gap function may be used to terminate conditional gradient
	methods.

	In the literature on numerical approximations of 
	optimization governed by differential equations, the focus is on studying
	the distance between approximate (local) solutions to those
	of the infinite dimensional problem
	\cite{Alt2012,Casas2012a,Casas2022,Dontchev2000,Dontchev1981,Ackooij2019,%
		Ackooij2019,Martens2023,Winkler2020}. 
	The distance-based errors quantify accuracy of (local) solutions.
	This amounts to establishing bounds on	the distance-based error measure
	\begin{align}
		\label{eq:intro-proximity}
		\|u_h^* - u^*\|_{U}.
	\end{align}
	Here $u^*$ is a (local) solution to the optimal control
	problem \eqref{eq:intro-infiniteproblem}
	and $u_h^*$ is a (local)  solution to its finite dimensional approximation
	\eqref{eq:intro-discretizedproblem}. 
	Our work augments the existing literature on error estimates for 
	infinite dimensional optimization in that we derive
	error estimates for criticality measures.
	Distance-based errors typically require
	strong convexity, second-order sufficient optimality, 
	or quadratic growth conditions 
	be satisfied. 	
	However, typical real-world applications lack strong convexity,
	such as the design of renewable tidal energy farms
	\cite{Funke2016,Piggott2022}. Moreover even for academic
	model problems,
	second-order sufficient 
	optimality conditions for infinite dimensional problems 
	remain computationally intractable to verify
	\cite{Roesch2008,Wachsmuth2009}.

	\section{Notation and preliminaries}
	
	We introduce some notation, and criticality measures 
	and discuss some of their properties and relationships.
	Many of these relationships are either essentially known
	or known for special cases. 
	We denote the gradient of a function $f$ by $\nabla f$
	and its  \gateaux\ derivative by $\Du f$.
	The norm of a Banach space $X$ is denoted by
	$\|\cdot\|_X$. For a bounded domain $\domain \subset \mathbb{R}^d$, 
	$L^p(\domain)$ $(1 \leq p \leq \infty)$ denote the standard
	Lebesgue spaces.
	We denote by $H^1(\domain)$ the space of weakly
	differentiable functions,
	and by $H_0^1(\domain)$
	the space of weakly differentiable functions with 
	zero boundary traces. We define the seminorm
	$|\cdot|_{H^1(\domain)} \coloneqq \|\nabla \cdot \|_{L^2(\domain)^d}$
	on $H^1(\domain)$ and equip
	$H^1(\domain)$ with the norm 
	$\|\cdot\|_{H^1(\domain)} \coloneqq 
	\sqrt{|\cdot|_{H^1(\domain)}^2+\|\cdot\|_{L^2(\domain)}^2}$.
	We equip $H_0^1(\domain)$ with 
	$|\cdot|_{H^1(\domain)}$.
	For $a$, $b \in L^2(\domain)$, we define the $L^2(\domain)$-box
	$[a,b]$ by 
	$[a,b] \coloneqq \{\, u \in L^2(\domain) \colon \, a \leq u \leq b\,\}$.
	The dual to a Banach space $X$ is denoted by $X^*$.
	
	Let $H$ be a real Hilbert space with
	inner product $(\cdot, \cdot)_H$
	and norm $\|\cdot\|_H \coloneqq \sqrt{(\cdot, \cdot)_H}$.
	Let $\phi : H \to (-\infty, \infty]$ be proper, closed, and convex.
	Its domain $\mathrm{dom}(\phi)$
	is  given by
	$\mathrm{dom}(\phi) \coloneq \{ u \in H \colon \phi(u) < \infty\}$.
	The proximity operator $\prox{\phi}{} : H \to (-\infty,\infty]$
	of $\phi$ is defined by
	\begin{align*}
		\prox{\phi}{u} \coloneqq   \mathrm{argmin}_{v \in H}\, 
		(1/2)\|v-u\|_H^2 + \phi(v).
	\end{align*}
	We denote the indicator function of 
	a closed, convex, nonempty set $H_0 \subset H$ by $I_{H_0}$,
	and the projection  $\proj{H_0}{}$
	onto $H_0$ by $\proj{H_0}{} \coloneqq \prox{I_{H_0}}{}$.

	We consider the composite optimization problem
	\begin{align}
		\label{eq:fpsi}
		\min_{u \in H}\, f(u) + \psi(u).
	\end{align}
	
	We impose mild conditions on the problem data $H$, $f$, and $\psi$.
	
	\begin{assumption}
		\label{assumption:preliminaries}
		The space $H$ is a real Hilbert space. 
		The function
		$\psi : H \to (-\infty, \infty]$ is proper, closed, and convex.
		The set 
		$V_H \subset H$ is open  with $\mathrm{dom}(\psi) \subset V_H$
		and $f : V_H \to \mathbb{R}$ is \gateaux\ differentiable.
	\end{assumption}
	
	Let \Cref{assumption:preliminaries} hold true.
	We refer to $\bar u \in V_H$ as critical point of \eqref{eq:fpsi}
	if $\bar u \in \mathrm{dom}(\psi)$ and
	$- \nabla f(\bar u) \in \partial \psi(\bar u)$.
	If $\bar u \in V_H$ is a local solution to \eqref{eq:fpsi}, then
	$\bar u$ is a critical point of \eqref{eq:fpsi}; cf.\
	\cite[p.\ 110]{Ito2008}. For $\tau > 0$,
	we define the normal map-based criticality measure 
	$\chi_{\text{nor}}(\cdot;\tau) : V_H \to [0,\infty)$ by
	\begin{align}
		\label{eq:chinor}
		\chi_{\text{nor}}(v;\tau) \coloneqq  
		\|\tau(v-\prox{\psi/\tau}{v}) + \nabla f(\prox{\psi/\tau}{v})\|_H.
	\end{align}
	Normal maps have been introduced in \cite{Robinson1992}.
	\Cref{lem:chinor} collects a few facts about the mapping
	$\chi_{\text{nor}}(\cdot, \tau)$. 
	\begin{lemma}
		\label{lem:chinor}
		Let \Cref{assumption:preliminaries} hold.
		Fix $\tau > 0$.
		\textup{(a)}
		If $\bar v \in H$ and $\chi_{\text{nor}}(\bar v;\tau) = 0$, then
		$\bar u \coloneqq \prox{\psi/\tau}{\bar v}$ is a critical point of
		\eqref{eq:fpsi}.
		\textup{(b)} If  $\bar u \in V_H$ is a critical point of
		\eqref{eq:fpsi}, then $\bar v \coloneqq \bar u - (1/\tau) \nabla f(\bar u)$
		satisfies $\chi_{\text{nor}}(\bar v;\tau) = 0$.
	\end{lemma}
	\begin{proof}
		For $V_H = H$, a proof is provided in 
		Proposition~3.5 in \cite{Pieper2015};
		see also Lemma~2.1 in \cite{Ouyang2021}
		where  $H = \mathbb{R}^n$ is considered. 
		These proofs can be adapted to our somewhat more general setting without
		any technical issues. The proof uses a fundamental characterization
		of the proximity operator provided in 
		Proposition~12.26 in \cite{Bauschke2011}.
	\end{proof}
	
	For $\tau > 0$, we consider the ``canonical'' criticality measure 
	$\chi_{\text{can}}(\cdot; \tau) : V_H \to [0,\infty)$ defined by
	\begin{align}
		\label{eq:canmeasure}
		\chi_{\text{can}}(u; \tau)
		\coloneqq  \|u-\prox{\psi/\tau}{u-(1/\tau)\nabla f(u)}\|_H.
	\end{align}
	This terminology is adapted from \cite{Milzarek2023}.
	The normal map-based criticality measure provides an upper bound
	on 	the canonical critically measure  in the following sense. 
	A similar result is provided in \cite[Lem.\ 4.1.6]{Milzarek2016}.
	\begin{lemma}
		\label{lem:chican_chinor}
		If \Cref{assumption:preliminaries} holds,
		$\tau > 0$, $v \in H$, and $u \coloneqq \prox{\psi/\tau}{v}$, then
		$\chi_{\text{can}}(u; \tau) \leq 
		(1/\tau)\chi_{\text{nor}}(v; \tau)$.
	\end{lemma}
	\begin{proof}
		We adapt computations in
		\cite[Cor.\ 6.3]{Milz2022b}.
		We have $u \in \mathrm{dom}(\psi)$. Hence
		$u \in V_H$.
		Since  $\prox{\psi/\tau}{}$ is firmly
		nonexpansive \cite[Prop.\ 12.28]{Bauschke2011}, we obtain
		\begin{align*}
			\chi_{\text{can}}(u; \tau)
			&= \|\prox{\psi/\tau}{v}-\prox{\psi/\tau}{u-(1/\tau)\nabla f(u)}\|_H
			\leq \|v-u+(1/\tau)\nabla f(u)\|_H
			\\
			& = (1/\tau) \|\tau(v-u)+\nabla f(u)\|_H
			= (1/\tau)\chi_{\text{nor}}(v; \tau).
		\end{align*}
	\end{proof}
	
	We define
	the  (regularized) 
	gap function
	$\chi_{\text{rgap}}(\cdot;\tau) \colon V_H \to (-\infty,\infty]$ by
	\begin{align}
		\label{eq:gapfunctional}
		\chi_{\text{rgap}}(u;\tau)
		\coloneqq   \sup_{v \in \mathrm{dom}(\psi)}\, 
		\big\{
		(\nabla f(u), u-v)_{H} + \psi(u)- \psi(v) - (\tau/2) \|u-v\|_H^2
		\big\},
	\end{align}
	where $\tau \geq 0$.
	For $\tau  = 0$, \eqref{eq:gapfunctional} becomes
	the unregularized gap function
	\cite{Hearn1982,Beck2015,Kunisch2022}.

	The function $\chi_{\text{rgap}}(\cdot;\tau)$ 
	measures approximate criticality
	of a point $\bar u \in V_H$  as the following simple
	fact demonstrates.
	
	\begin{lemma}
		\label{lem:chivarbound}
		Under \Cref{assumption:preliminaries}
		Let $\tau \geq 0$, and let $\tau \geq 0$,
		\textnormal{(a)} $\chi_{\text{rgap}}(u;\tau) = \infty$
		if $u \in V_H \setminus \mathrm{dom}(\psi)$,
		\textnormal{(b)}
		$\chi_{\text{rgap}}(u;0)
		\geq \chi_{\text{rgap}}(u;\tau)\geq 0$ for all $u \in V_H$, and 
		\textnormal{(c)}
		if $\varepsilon \in [0,\infty)$ and $u \in V_H$, then
		$\chi_{\text{rgap}}(u;\tau)\leq \varepsilon$ if and only if
		$u \in \mathrm{dom}(\psi)$ and
		\begin{align*}
			(\nabla f(u), v-u)_{H} + \psi(v)- \psi(u)
			\geq -(\tau/2) \|u-v\|_H^2 -\varepsilon
			\quad \text{for all} \quad v \in H.
		\end{align*}
	\end{lemma}
	\begin{proof}
		\textnormal{(a)} Since $\mathrm{dom}(\psi)$ is nonempty, there exists
		$v \in \mathrm{dom}(\psi)$. Using \eqref{eq:gapfunctional}, 
		we obtain
		$\chi_{\text{rgap}}(u) = \infty$.
		
		\textnormal{(b)} 
		The first inequality is a direct consequence of
		\eqref{eq:gapfunctional}.
		If $u \not\in \mathrm{dom}(\psi)$, then 
		$\chi_{\text{rgap}}(u;\tau) = \infty \geq 0$.
		If $u \in \mathrm{dom}(\psi)$, then choosing
		$v = u$ in \eqref{eq:gapfunctional} shows that
		$\chi_{\text{rgap}}(u;\tau) \geq 0$.
		
		\textnormal{(d)} If 	$\chi_{\text{rgap}}(u;\tau) \leq \varepsilon$,
		then \eqref{eq:gapfunctional} 
		and the fact that $\mathrm{dom}(\psi)$
		is nonempty ensure $u \in \mathrm{dom}(\psi)$.
		Rearranging terms yields the assertions.
	\end{proof}
	
	The gap function provides an upper bound on squared
	canonical criticality measure.
	An estimate related to that in \eqref{eq:chicanchivar} is
	established in \cite[sect.\ 7.5.1]{Lan2020}.
	
	\begin{lemma}
		\label{lem:chicanchivar}
		Let $\tau > 0$, and let $\nu \geq 0$.
		If \Cref{assumption:preliminaries} holds and
		$u \in V_H$, then
		\begin{align}
			\label{eq:chicanchivar}
			(\tau-\nu/2)\chi_{\text{can}}(u; \tau)^2 \leq 
			\chi_{\text{rgap}}(u;\nu).
		\end{align}
	\end{lemma}
	
	\begin{proof}[{Proof of \Cref{lem:chicanchivar}}]
		We define 
		$v \coloneqq u-(1/\tau)\nabla f(u)$
		and $w \coloneqq \prox{\psi/\tau}{v}$.
		Using Proposition~12.26 in \cite{Bauschke2011}, we have
		\begin{align*}
			(	u-w, v-w)_H + \psi(w)/\tau \leq \psi(u)/\tau.
		\end{align*}
		Combined with \eqref{eq:canmeasure},
		\begin{align*}
			\chi_{\text{can}}(u; \tau)^2 = \|u-w\|_H^2 
			= (u-w, v-w)_H + (u-w, u-v)_H.
		\end{align*}
		We obtain
		\begin{align*}
			\|u-w\|_H^2
			&= (u-w, v-w)_H + (1/\tau)(u-w, \nabla f(u))_H
			\\
			& \leq  (1/\tau)\psi(u) - (1/\tau)\psi(w)
			+ (1/\tau)(\nabla f(u),u-w)_H.
		\end{align*}
		Combined with \eqref{eq:gapfunctional}, we have
		\begin{align*}
			(\tau-\nu/2)\|u-w\|_H^2
			\leq   \psi(u) - \psi(w)
			+ (\nabla f(u),u-w)_H
			- (\nu/2)\|u-w\|_H^2
			\leq \chi_{\text{rgap}}(u;\nu).
		\end{align*}
		Using \eqref{eq:canmeasure},
		we obtain \eqref{eq:chicanchivar}. 
	\end{proof}
	
	\section{Infinite dimensional
		optimization problems and their discretizations}
	
	We consider
	\begin{align}
		\label{eq:infiniteproblem}
		\min_{u \in U}\, \widehat{J}(u) + \psi(u)
	\end{align}
	and its discretization given by
	\begin{align}
		\label{eq:discretizedproblem}
		\min_{u_h \in U_h}\, \widehat{j}_h(u_h) + \psi_h(u_h).
	\end{align}
	
	We define the function $\widehat{J}$  in \eqref{eq:trueobjective}
	after formulating \Cref{assumption:trueproblem}
	and $\widehat{j}_h$ in \eqref{eq:discretizedobjective} after
	stating \Cref{assumption:discretizedproblem}.

	\begin{assumption}[{infinite dimensional optimization problem}]
		\label{assumption:trueproblem}
		~
		\begin{enumerate}[nosep]
			\item $U$ is a real, separable Hilbert space, and $Y$ 
			is real, separable Banach spaces.
			\item $\psi : U \to (-\infty,\infty]$
			is proper, closed, and convex.
			\item $J : Y \times U \to \mathbb{R}$ 
			is continuously differentiable.
			\item  $V_U \subset U$ is an open 
			set with 	$\mathrm{dom}(\psi) \subset V_U$.
			\item $S \colon V_U \to Y$ is continuously differentiable.
		\end{enumerate}
	\end{assumption}

	Let \Cref{assumption:trueproblem} hold true.
	We define $\widehat{J} : V_U \to \mathbb{R}$ by
	\begin{align}
		\label{eq:trueobjective}
		\widehat{J}(u) \coloneqq J(S(u),u).
	\end{align}
	
	We state conditions allowing us to formulate finite dimensional
	approximations to \eqref{eq:infiniteproblem}. 
	Our approach to discretize \eqref{eq:infiniteproblem} is mainly inspired
	by the text  \cite{Hintermueller2004}
	(see also
	\cite{Arnautu1998,Falk1973,Geveci1979,%
		Malanowski1982,Meidner2008,Neittaanmaki2006,%
		Grigorieff1990,Grigorieff1990a}).
	
	\begin{assumption}[{discretized optimization problem}]
		\label{assumption:discretizedproblem}
		Let $U$, $Y$, and $V_U$ be as in 
		\Cref{assumption:trueproblem}. 
		For each $h > 0$, it holds that
		\begin{enumerate}[nosep]
			\item $U_h$ and $Y_h$  are subspaces
			of $U$ and $Y$, respectively.
			\item $\psi_h : U_h \to (-\infty,\infty]$
			is proper, closed, and convex.
			\item $\mathrm{dom}(\psi_h) \subset V_U$.
			\item $S_h \colon V_U \to Y_h$ is continuously differentiable.
		\end{enumerate}
	\end{assumption}

	We define $\iota_{U_h,U}\colon U_h \to U$ by 
	$\iota_{U_h,U}u_h \coloneqq  u_h$,
	and $\iota_{Y_h,Y} : Y_h \to Y$ by
	$\iota_{Y_h,Y}y_h \coloneqq  y_h$.
	These mappings are linear, have operator norm one
	and hence are bounded. For the finite dimensional problem
	\eqref{eq:discretizedproblem}, we define
	$\widehat{j}_h : U_h \cap V_U \to \mathbb{R}$
	and  $\widehat{J}_{h} : V_U \to \mathbb{R}$ by
	\begin{align}
		\label{eq:discretizedobjective}
		\widehat{j}_h(u_h) \coloneqq \widehat{J}_h(\iota_{U_h,U}u_h),
		\quad \text{and} \quad 
		\widehat{J}_{h}(u) \coloneqq J(\iota_{Y_h,Y}S_h(u), u).
	\end{align}
	Making explicit the operator $\iota_{U_h,U}$
	in \eqref{eq:discretizedobjective} helps us computing derivatives and
	gradients.
	Fix $u_h \in U_h \cap V_U$.
	Using the chain rule, we find that
	\begin{align*}
		\mathrm{D}_{u_h}\widehat{j}_h(u_h) 
		= \iota_{U_h,U}^*\mathrm{D}_u \widehat{J}_h(\iota_{U_h,U}u_h).
	\end{align*}
	Let $R_{U_h} : U_h^* \to U_h$ 
	and $R_{U} : U^* \to U$  be Riesz mappings.
	We obtain
	\begin{align*}
		\begin{aligned}
			\nabla_{u_h} \widehat{j}_h(u_h) 
			= R_{U_h}\mathrm{D}_{u_h}\widehat{j}_h(u_h) 
			= R_{U_h}\iota_{U_h,U}^*\mathrm{D}_u 
			\widehat{J}_h(\iota_{U_h,U}u_h)
			= R_{U_h}\iota_{U_h,U}^*R_{U}^{-1} 
			\nabla_u \widehat{J}_h(\iota_{U_h,U}u_h).
		\end{aligned}
	\end{align*}
	
	Let us define $\Pi_h : U \to U_h$ by
	$\Pi_h \coloneqq R_{U_h}\iota_{U_h,U}^*R_U^{-1}$.
	We obtain the gradient formula
	\begin{align}
		\label{eq:nablajh}
		\begin{aligned}
			\nabla_{u_h} \widehat{j}_h(u_h) 
			= \Pi_h \nabla_u \widehat{J}_h(\iota_{U_h,U}u_h)
			= \Pi_h \nabla_u \widehat{J}_h(u_h).
		\end{aligned}
	\end{align}
	The operator $\Pi_h$
	is the Hilbert adjoint operator
	$\iota_{U_h,U}^\star$ of $\iota_{U_h,U}$
	\cite[p.\ 237]{Kreyszig1978}, 
	that is, $\iota_{U_h,U}^\star = \Pi_h$.
	Hence for all $v \in U$ and $w_h \in U_h$,
	\begin{align}
		\label{eq:projectionontoUh}
		(\Pi_h v, w_h)_U  = 
		(\iota_{U_h,U}^\star v, w_h)_U  
		=  (v,\iota_{U_h,U}w_h)_U
		= (v,w_h)_U.
	\end{align}
	Therefore $\Pi_h$ is the (orthogonal) projection  onto $U_h$
	and we have 
	$\|\Pi_h v\|_{U} \leq \|v\|_U$ for all $v \in U$
	and $h  > 0$ \cite[Lem.\ 9.18]{Alt2016}.

	\begin{remark}
		\Cref{assumption:trueproblem,assumption:discretizedproblem}
		allow for the modeling choice $S_h = S$
		with $Y_h = Y$.
	\end{remark}
	
	\Cref{assumption:trueproblem,assumption:discretizedproblem}
	are often satisfied for PDE-constrained optimization.

	\section{Error estimates}

	Nonasymptotic error estimates require ``higher regularity''
	\cite{Sayas2004}. We formulate ``higher regularity'' in the next 
	two assumptions.

	\begin{assumption}
		\label{assumption:gelfand-projection}
		Let $U$ be as in \Cref{assumption:trueproblem}. 
		\begin{enumerate}[nosep]
			\item The space $\widehat{U}$ is a real, reflexive Banach space, $\widehat{U} \subset U$,
			and the operator  $\iota_{\widehat{U},U} \colon \widehat{U} \to U$
			defined by  $\iota_{\widehat{U},U}w \coloneqq w$ is  continuous
			and its image is dense in $U$.
			\item For a nondecreasing function 
			$\rho_{\Pi} : [0,\infty) \to [0,\infty)$
			with $\rho_{\Pi}(0) = 0$,
			\begin{align*}
				\|\Pi_h v - v\|_U \leq \rho_{\Pi}(h)\|v\|_{\widehat{U}}
				\quad \text{for all} \quad v \in \widehat{U},
				\quad h > 0.
			\end{align*}
		\end{enumerate}
	\end{assumption}
	
	\Cref{assumption:gelfand-projection} is fulfilled
	if $U = L^2(\domain)$, $\widehat{U} = H^1(\domain)$,
	and $U_h$ is the space of piecewise constant functions
	\cite[Prop.\ 1.135]{Ern2004}
	or if $U = L^2(\domain)$, $\widehat{U} = H^2(\domain)$,
	and $U_h$ is the space of piecewise linear continuous functions
	\cite[Prop.\ 1.134]{Ern2004}, for example.
	
	Upon identifying $U^*$ with $U$ and writing $U^* = U$,
	\Cref{assumption:gelfand-projection} ensures that the triple
	$\widehat{U} \hookrightarrow U \hookrightarrow \widehat{U}^*$ is a Gelfand triple
	\cite[Def.\ 17.1]{Wloka1987}.

	If the spaces $U_h$ are finite dimensional, then 
	\Cref{assumption:gelfand-projection} implies that the embedding
	$\widehat{U} \hookrightarrow U$ is compact, that is, 
	$\iota_{\widehat{U},U}$ is a compact operator. Although not explicitly used 
	within our error estimation, this fact highlights some ``hidden''
	compactness resulting from an approximation accuracy.

	\begin{lemma}
		If \Cref{assumption:gelfand-projection} holds,
		and for each $h > 0$, $U_h$ is finite dimensional,
		then $\iota_{\widehat{U},U}$ is compact.
	\end{lemma}
	\begin{proof}
		\Cref{assumption:gelfand-projection}  ensures
		$\|\Pi_h \iota_{\widehat{U},U} v - \iota_{\widehat{U},U} v\|_U 
		\leq \rho_{\Pi}(h)\|v\|_{\widehat{U}}$
		for all $h > 0$ and  $v \in \widehat{U}$.
		Hence $\Pi_h \iota_{\widehat{U},U} $ converges uniformly 
		to $\iota_{\widehat{U},U}$ as $h \to 0^+$. Since $U_h$ is finite dimensional,
		the linear bounded operator $\Pi_h \iota_{\widehat{U},U}$ is compact
		\cite[Thm.\ 8.1-4]{Kreyszig1978}. Hence
		$\iota_{\widehat{U},U}$ is compact  \cite[Thm.\ 8.1-5]{Kreyszig1978}.
	\end{proof}
	
	Having introduced $\widehat{U}$, we can impose ``higher regularity''
	of the gradient $\nabla \widehat{J}$.
	
	\begin{assumption}
		\label{assumption:uniformbounded}
		Let $V_U$ be as in \Cref{assumption:trueproblem},
		and let $U_h$ be as in \Cref{assumption:discretizedproblem}.
		\begin{enumerate}[nosep]
			\item For a bounded set $W_U\subset V_U$, 
			$\mathrm{dom}(\psi) \subset W_U$,
			and $\mathrm{dom}(\psi_h) \subset W_U$
			for all $h > 0$.
			\item The gradient $\nabla \widehat{J}$ is Lipschitz continuous
			on $W_U$ with Lipschitz constant $L_{\nabla \widehat{J}} > 0$.
			\item For all $h > 0$ and $v_h \in W_U \cap U_h$,
			$\nabla \widehat{J}(v_h) \in \widehat{U}$,
			and there exists $\ell_{\nabla \widehat{J}} > 0$ such that
			\begin{align}
				\label{eq:gradientbound}
				\|\nabla \widehat{J}(v_h)\|_{\widehat{U}} 
				\leq \ell_{\nabla \widehat{J}}
				\quad \text{for all} \quad v_h \in W_U \cap U_h, \quad 
				\quad h  > 0.
			\end{align}
			\item For a nondecreasing function 
			$\rho_{\nabla \widehat{J}} : [0,\infty) \to [0,\infty)$
			with $\rho_{\nabla \widehat{J}}(0) = 0$,
			\begin{align}
				\label{eq:gradienterror}
				\|\nabla \widehat{J}_{h}(v_h)- \nabla\widehat{J}(v_h)\|_U
				\leq \rho_{\nabla \widehat{J}}(h)
				\quad \text{for all} \quad v_h \in W_U \cap U_h, \quad 
				\quad h  > 0.
			\end{align}
		\end{enumerate}
	\end{assumption}

	\subsection{Normal map-based error estimates}

	We define the normal map-based criticality measure
	\begin{align}
		\label{eq:normap}
		\chi_{\text{nor}}(v;\tau) & \coloneqq  
		\|\tau(v-\prox{\psi/\tau}{v}) + \nabla \widehat{J}(\prox{\psi/\tau}{v})\|_U,
	\end{align}
	and its  approximation
	\begin{align}
		\label{eq:normaph}
		\chi_{\text{nor}, h}(v_h;\tau) & \coloneqq   
		\|\tau(v_h-\prox{\psi_h/\tau}{v_h}) + 
		\nabla \widehat{j}_h(\prox{\psi_h/\tau}{v_h})\|_U.
	\end{align}
	
	One relationship between
	the normal map-based criticality measure and its approximation
	is provided by the following.
	If $v_h \in U_h$ satisfies $\chi_{\text{nor}, h}(v_h;\tau)  =0$, 
	and $\psi_h = \psi$ on $U_h$, then 
	\eqref{eq:normap} and \eqref{eq:normaph}  ensure
	\begin{align*}
		\chi_{\text{nor}}(v_h;\tau) =
		\|\nabla \widehat{J}(\prox{\psi/\tau}{v_h})
		- \nabla \widehat{j}_h(\prox{\psi/\tau}{v_h})\|_U.
	\end{align*}
	
	The next assumption quantifies the difference between
	the proximity operators of $\psi/\tau$ and $\psi_h/\tau$.
	
	\begin{assumption}
		\label{assumption:prox}
		Let $U_h$ be as in \Cref{assumption:discretizedproblem}.
		For a nondecreasing function 
		$\rho_{\mathrm{prox}} : [0,\infty) \to [0,\infty)$
		with $\rho_{\mathrm{prox}}(0) = 0$,
		\begin{align*}
			\|\prox{\psi/\tau}{v_h}-\prox{\psi_h/\tau}{v_h}\|_U
			\leq \rho_{\mathrm{prox}}(h)
			\quad \text{for all} 
			\quad v_h \in U_h, \quad h > 0.
		\end{align*}
	\end{assumption}
	
	If $\psi_h = \psi$ on $U_h$
	for each $h > 0$, then \Cref{assumption:prox} hold true
	with $\rho_{\mathrm{prox}} = 0$.

	We state our main result.
	
	\begin{theorem}
		\label{thm:normbased}
		If $\tau >0 $, and 
		\Cref{assumption:trueproblem,assumption:discretizedproblem,%
			assumption:gelfand-projection,assumption:prox,assumption:uniformbounded} hold, then for all $h > 0$ and $\bar v_h \in W_U \cap U_h$, 
		\begin{align*}
			\chi_{\text{nor}}(\bar v_h;\tau)
			\leq \chi_{\text{nor},h}(\bar v_h;\tau) 
			+ (\tau + L_{\nabla \widehat{J}})
			\rho_{\mathrm{prox}}(h)
			+ \rho_{\nabla \widehat{J}}(h) 
			+\rho_{\Pi}(h)\ell_{\nabla \widehat{J}}.
		\end{align*}
	\end{theorem}
	
	We prepare our proof of \Cref{thm:normbased}. 
	
	\begin{lemma}
		\label{prop:normbasedbasicestimate}
		Let \Cref{assumption:trueproblem,assumption:discretizedproblem}  hold,
		and let $v_h \in U_h \cap V_U$. For
		$u_h \coloneqq \prox{\psi_h/\tau}{v_h}$
		and
		$\bar u_h \coloneqq \prox{\psi/\tau}{v_h}$, we have
		\begin{align*}
			\|\tau(v_h -\bar u_h) 
			+ \nabla \widehat{J}(\bar u_h)\|_U
			& \leq 	\|\tau(v_h -u_h) + \nabla \widehat{j}_h(u_h)\|_U 
			+ \tau \|\bar u_h -u_h\|_U
			\\
			& \quad + \| \nabla \widehat{J}(\bar u_h)- \nabla \widehat{J}(u_h)\|_U
			\\
			& \quad + \|\nabla \widehat{J}_{h}(u_h)-\nabla \widehat{J}(u_h)\|_U
			+
			\|\nabla \widehat{J}(u_h)-\Pi_h \nabla \widehat{J}(u_h)\|_U.
		\end{align*}
	\end{lemma}
	\begin{proof}
		We have $u_h \in \mathrm{dom}(\psi_h) \subset U_h \cap V_U$
		and $\bar u_h \in \mathrm{dom}(\psi) \subset V_U$.
		Using the triangle inequality, we have
		\begin{align*}
			\|\tau(v_h -\bar u_h) 
			+ \nabla \widehat{J}(\bar u_h)\|_U
			& 
			\leq 
			\|\tau(v_h -u_h) + \nabla \widehat{J}(u_h)\|_U
			+ \| \nabla \widehat{J}(\bar u_h)- \nabla \widehat{J}(u_h)\|_U
			\\
			& \quad + \tau \|\bar u_h -u_h\|_U.
		\end{align*}
		We further estimate
		\begin{align*}
			\|\tau(v_h - u_h) + \nabla \widehat{J}(u_h)\|_U
			& \leq 
			\|\tau(v_h - u_h) 
			+ \nabla \widehat{j}_h(u_h)\|_U
			+
			\|\nabla \widehat{j}_h(u_h)-\Pi_h \nabla \widehat{J}(u_h)\|_U
			\\
			& \quad + \|\nabla \widehat{J}(u_h)-\Pi_h \nabla \widehat{J}(u_h)\|_U.
		\end{align*}
		
		We have $\nabla \widehat{j}_h(u_h) = \Pi_h \nabla \widehat{J}_h(u_h)
		$
		(see \eqref{eq:nablajh}) and the operator norm
		of $\Pi_h$ is one. Hence
		\begin{align}
			\label{eq:Feb720241153}
			\|\nabla \widehat{j}_h(u_h)-\Pi_h \nabla \widehat{J}(u_h)\|_U
			\leq 
			\|\nabla \widehat{J}_h(u_h)-\nabla \widehat{J}(u_h)\|_U.
		\end{align}
		Putting together the pieces, 
		we obtain the assertion.
	\end{proof}
	
	\begin{proof}[{Proof of \Cref{thm:normbased}}]
		Using \Cref{prop:normbasedbasicestimate} and
		\Cref{assumption:trueproblem,assumption:discretizedproblem,%
			assumption:gelfand-projection,assumption:prox,assumption:uniformbounded},
		we obtain the assertion.
	\end{proof}

	We show that the convergence rates provided by
	\Cref{thm:normbased} are order-optimal.
	
	\begin{example}
		\label{example:normbased}
		Let 	
		$
		\xi(x) \coloneqq -x
		$, 
		and $\mathbf{1}(x) \coloneqq 1$.
		Let $\psi$ be the indicator function 
		of the
		box $\{u \in L^2(0,1) \colon  \xi \leq u \leq \mathbf{1}\}$.
		We consider  the infinite dimensional linear program
		\begin{align*}
			\min_{u \in L^2(0,1)}\, (\mathbf{1}, u)_{L^2(0,1)} + \psi(u). 
		\end{align*}
		
		We discretize $L^2(0,1)$ using the space of piecewise
		constant functions $U_h$ defined on an equidistant grid 
		$((j-1)/n, j/n)$ for $j = 1, \ldots, n$ of $[0,1]$
		with  $n \in \mathbb{N}$ being the number of grids. 
		We define $h \coloneqq 1/n$. 
		On each interval $((j-1)/n, j/n)$,
		the $L^2(0,1)$-orthogonal projection 
		$\Pi_h \colon L^2(0,1) \to U_h$ evaluated at $u \in L^2(0,1)$
		equals the average of $u$ on each of the intervals 
		\cite[eq.\ (1.123)]{Ern2004}.
		\Cref{assumption:gelfand-projection} is satisfied with 
		$\widehat{U} \coloneqq H^1(0,1)$
		and $\rho_{\Pi}(h)$ linear in $h$
		\cite[Prop.\ 1.135]{Ern2004}.
		Because of $\nabla \widehat{J}(u) = \mathbf{1}$
		for all $u \in L^2(\domain)$,
		we can choose $\rho_{\nabla \widehat{J}}$
		to be the zero function in  \eqref{eq:gradienterror}. 
		According to \Cref{lem:canonicalpsih},
		\Cref{assumption:prox} holds true with 
		$\rho_{\mathrm{prox}}$ linear in $h$.

		Let $\psi_h$ be the indicator function of 
		$\{u \in U_h\colon \Pi_h\xi \leq u \leq \mathbf{1}\}$.
		We consider the canonical finite dimensional approximation
		\begin{align*}
			\min_{u_h \in U_h}\, (\mathbf{1}, u_h)_{L^2(0,1)} + \psi_h(u_h). 
		\end{align*}
		
		We choose $\tau \coloneqq 1$.
		We define $u_h^* \coloneqq \Pi_h \xi$
		and $v_h^* \coloneqq u_h^* - \mathbf{1}$.
		Since $u_h^*$ solves the finite dimensional problem,
		\Cref{lem:chinor} yields
		$\chi_{\text{nor}, h}(v_h^*;\tau)  =0$.
		We have $\prox{\psi_h}{v_h^*} = \Pi_h \xi$, 
		and $\prox{\psi}{v_h^*} = \xi$.
		Using \eqref{eq:normap} and \eqref{eq:normaph}, 
		$\nabla \widehat{J}(\prox{\psi}{v_h^*})
		- \nabla \widehat{j}_h(\prox{\psi_h}{v_h^*}) = 0$,
		and
		$\|\xi - \Pi_h \xi\|_{L^2(0,1)} = n \cdot (h^2/4) = h/4$, we have
		\begin{align*}
			\chi_{\text{nor}}(v_h^*;1)  
			= \|\prox{\psi_h}{v_h^*} - \prox{\psi}{v_h^*}\|_{L^2(0,1)}
			=
			\|\xi - \Pi_h \xi\|_{L^2(0,1)} = (1/4) h.
		\end{align*}
	\end{example}

	\subsection{Canonical critically measure-based error estimates}

	We define the canonical criticality measure
	\begin{align*}
		\chi_{\text{can}}(u; \tau)
		& \coloneqq \|u-\prox{\psi/\tau}{u-(1/\tau)\nabla \widehat{J}(u)}\|_U,
	\end{align*}
	and its approximation
	\begin{align*}
		\chi_{\text{can},h}(u_h; \tau)
		& \coloneqq  \|u_h-\prox{\psi_h/\tau}{u_h-(1/\tau)\nabla \widehat{j}_h(u_h)}\|_U.
	\end{align*}
	
	\Cref{lem:chican_chinor} ensures 
	$\chi_{\text{can}}(\bar u_h; \tau) 
	\leq (1/\tau) \chi_{\text{nor}}(\bar v_h;\tau)$,
	where 
	$\bar u_h = \prox{\psi/\tau}{\bar v_h}$ and $\bar v_h \in U_h$.
	Combined with \Cref{thm:normbased}, we obtain error bounds on
	$\chi_{\text{can}}(\bar u_h; \tau)$.
	The control $\bar u_h = \prox{\psi/\tau}{\bar v_h}$ can be computed
	within a postprocessing step. 
	
	We establish error bounds 
	for the difference $\chi_{\text{can}}(u_h; \tau)
	- \chi_{\text{can},h}(u_h; \tau) $
	for $u_h \in U_h$.
	
	\begin{theorem}
		\label{thm:canbased}	
		If $\tau >0 $, and 
		\Cref{assumption:trueproblem,assumption:discretizedproblem,%
			assumption:gelfand-projection,assumption:prox,assumption:uniformbounded} hold, then for all $h > 0$ and $u_h \in W_U \cap U_h$, 
		\begin{align*}
			\chi_{\text{can}}(u_h; \tau)
			\leq \chi_{\text{can},h}(u_h; \tau) +
			(1/\tau)  \rho_{\nabla \widehat{J}}(h) + \rho_{\mathrm{prox}}(h)
			+(1/\tau) \rho_{\Pi}(h)\ell_{\nabla \widehat{J}}.
		\end{align*}	
	\end{theorem}
	
	\begin{proof}
		The triangle inequality yields
		\begin{align*}
			\chi_{\text{can}}(u_h; \tau)
			& = \|u_h-\prox{\psi/\tau}{u_h-(1/\tau)\nabla \widehat{J}(u_h)}\|_U
			\\
			& \leq 
			\|u_h-\prox{\psi_h/\tau}{u_h-(1/\tau)\nabla \widehat{j}_h(u_h)}\|_U
			\\
			& \quad + 
			\|\prox{\psi_h/\tau}{u_h-(1/\tau) \Pi_h\nabla \widehat{J}(u_h)}
			-\prox{\psi_h/\tau}{u_h-(1/\tau)\nabla \widehat{j}_h(u_h)}\|_U
			\\
			& \quad + 
			\|\prox{\psi_h/\tau}{u_h-(1/\tau) \Pi_h\nabla \widehat{J}(u_h)}
			-\prox{\psi/\tau}{u_h-(1/\tau) \Pi_h \nabla \widehat{J}(u_h)}\|_U
			\\
			& \quad + 
			\|\prox{\psi/\tau}{u_h-(1/\tau) \nabla \widehat{J}(u_h)}
			-\prox{\psi/\tau}{u_h-(1/\tau) \Pi_h \nabla \widehat{J}(u_h)}\|_U.
		\end{align*}
		Since $\prox{\psi_h/\tau}{\cdot}$ is firmly nonexpansive, we
		obtain 
		\begin{align*}
			& \|\prox{\psi_h/\tau}{u_h-(1/\tau) \Pi_h\nabla \widehat{J}(u_h)}
			-\prox{\psi_h/\tau}{u_h-(1/\tau)\nabla \widehat{j}_h(u_h)}\|_U
			\\
			& \leq  (1/\tau) 
			\|\Pi_h\nabla \widehat{J}(u_h)
			-\nabla \widehat{j}_h(u_h)\|_U.
		\end{align*}
		Combined with  \eqref{eq:Feb720241153} and
		\Cref{assumption:uniformbounded}, we obtain
		\begin{align*}
			\|\prox{\psi_h/\tau}{u_h-(1/\tau) \Pi_h\nabla \widehat{J}(u_h)}
			-\prox{\psi_h/\tau}{u_h-(1/\tau)\nabla \widehat{j}_h(u_h)}\|_U
			\leq   (1/\tau)  \rho_{\nabla \widehat{J}}(h). 
		\end{align*}
		
		Since $\prox{\psi/\tau}{\cdot}$ is firmly nonexpansive, 
		\Cref{assumption:uniformbounded} ensures
		\begin{align*}
			& \|\prox{\psi/\tau}{u_h-(1/\tau) \nabla \widehat{J}(u_h)}
			-\prox{\psi/\tau}{u_h-(1/\tau) \Pi_h \nabla \widehat{J}(u_h)}\|_U
			\\
			& \leq
			(1/\tau) \| \nabla \widehat{J}(u_h)
			-\Pi_h \nabla \widehat{J}(u_h)\|_U
			\leq (1/\tau) \rho_{\Pi}(h)\ell_{\nabla \widehat{J}}.
		\end{align*}
		Combined with \Cref{assumption:prox}, we obtain the assertion.
	\end{proof}
	
	We show that the convergence rates in \Cref{thm:canbased} are order-optimal.
	
	\begin{example}
		We reconsider \Cref{example:normbased}. 
		We have 
		$\chi_{\text{can}}(u_h^*; 1) = \|\Pi_h\xi - \prox{\psi}{u_h^*-
			\mathbf{1}}\|_{L^2(0,1)} = \|\xi - \Pi_h \xi\|_{L^2(0,1)} = (1/4) h$.
		Combined with the findings in 
		\Cref{example:normbased}, 
		we find that the convergence rate in \Cref{thm:canbased} 
		cannot be improved in general.
	\end{example}

	\subsection{Gap function-based error estimates}
	\label{subsection:gapfunction}

	Throughout the section, we consider the composite functions
	\begin{align}
		\label{eq:compositephi}
		\psi(u) \coloneqq \phi(u) + I_{U_{\text{ad}}}(u),
		\quad \text{and} \quad 
		\psi_h(u_h) \coloneqq \phi(u_h) + I_{\Pi_h U_{\text{ad}}}(u_h),
	\end{align}
	where $\phi \colon U \to  \mathbb{R}$ is a real-valued convex
	function,  and $U_{\text{ad}} \subset U$ is a closed, convex, 
	bounded, nonempty set.
	We denote by $L_\phi$ the Lipschitz constant of $\phi$
	on the bounded set $W_U$ from 
	\Cref{assumption:uniformbounded}.

	We define 
	gap function $\chi_{\text{gap}} \colon V_U \to (-\infty,\infty]$  by
	\begin{align}
		\label{eq:chivar}
		\chi_{\text{gap}}(u)
		& \coloneqq   \sup_{v \in U_{\text{ad}} }\, \big\{\,
		(\nabla \widehat{J}(u), u-v)_{U} + \phi(u)- \phi(v)
		\,\big\}.
	\end{align}
	Its approximation
	$\chi_{\text{gap},h} \colon V_U \cap U_h \to (-\infty,\infty]$ 
	is defined by
	\begin{align}
		\label{eq:chivarh}
		\chi_{\text{gap},h}(u_h)
		& \coloneqq   \sup_{v_h \in \Pi_h U_{\text{ad}} }\, 
		\big\{\,
		(\nabla \widehat{j}_h(u_h), u_h-v_h)_{U} + \phi(u_h)- \phi(v_h)
		\,\big\}.
	\end{align}
	
	The composite functions \eqref{eq:compositephi}
	allow for  $\chi_{\text{gap}}(u_h) < \infty$ 
	even if $u_h \not\in\mathrm{dom}(\psi)$.
	If $u_h \in U_{\text{ad}}$, then 
	\Cref{lem:chicanchivar} yields
	$\tau \chi_{\text{can}}(u_h;\tau)^2 \leq \chi_{\text{gap}}(u_h)$
	for each $\tau >0$,
	providing a lower bound on $\chi_{\text{gap}}(u_h)$
	in terms of the squared canonical criticality measure.
	
	As a substitute of \Cref{assumption:prox}, we 
	formulate \Cref{assumption:proj}.
	
	\begin{assumption}
		\label{assumption:proj}
		Let $U_h$ be as in \Cref{assumption:discretizedproblem}.
		For a nondecreasing function 
		$\rho_{\mathrm{proj}} : [0,\infty) \to [0,\infty)$
		with $\rho_{\mathrm{proj}}(0) = 0$,
		\begin{align*}
			\|\proj{U_{\text{ad}}}{v_h}-\proj{\Pi_hU_{\text{ad}}}{v_h}\|_U
			\leq \rho_{\mathrm{proj}}(h)
			\quad \text{for all} 
			\quad v_h \in U_h, \quad h > 0.
		\end{align*}
	\end{assumption}
	
	We denote by $\mathrm{diam}(W_U)$ the diameter of
	$W_U$.
	
	\begin{theorem}
		\label{thm:gapbased}	
		Suppose that for all $h > 0$
		and $u \in U$,
		we have $\phi(\Pi_h u) \leq \phi(u)$. 
		If 
		\Cref{assumption:trueproblem,assumption:discretizedproblem,%
			assumption:gelfand-projection,assumption:proj,assumption:uniformbounded} hold, then for all $h > 0$ and $u_h \in W_U \cap U_h$
		with $\bar u_h \coloneqq \proj{U_{\text{ad}}}{u_h}$,
		\begin{align*}
			\chi_{\text{gap}}(\bar u_h)
			&\leq \chi_{\text{gap},h}(u_h)
			+ (\ell_{\nabla \widehat{J}} +
			L_\phi) \rho_{\mathrm{proj}}(h)
			\\
			& \quad + 
			\mathrm{diam}(W_U)
			\big[
			L_{\nabla \widehat{J}} \rho_{\mathrm{proj}}(h)+
			\rho_{\Pi}(h)\ell_{\nabla \widehat{J}}
			+ \rho_{\nabla \widehat{J}}(h)
			\big].
		\end{align*}
	\end{theorem}
	\begin{proof}
		Let $\bar v_h$ be a solution to \eqref{eq:chivar}
		with $u = \bar u_h$.
		Hence
		\begin{align*}
			\chi_{\text{gap}}(\bar u_h)
			& = 
			(\nabla \widehat{J}(
			\bar u_h), \bar u_h-\bar v_h)_{U} + \phi(\bar u_h)- \phi(\bar v_h).
		\end{align*}
		Using \eqref{eq:chivarh}, 
		and $(\nabla \widehat{j}_h(u_h), \Pi_h \bar v_h )_{U}
		= (\nabla \widehat{j}_h(u_h),  \bar v_h )_{U}$, 
		we obtain
		\begin{align*}
			\chi_{\text{gap},h}(u_h)
			\geq 
			(\nabla \widehat{j}_h(u_h), u_h-\bar v_h )_{U} + 
			\phi(u_h)- \phi(\Pi_h\bar v_h ).
		\end{align*}
		Hence
		\begin{align*}
			\chi_{\text{gap}}(u_h)
			& \leq \chi_{\text{gap},h}(u_h)
			+ (\nabla \widehat{J}(\bar u_h)-\nabla \widehat{j}_h(u_h), u_h-\bar v_h)_{U} +  \phi(\Pi_h\bar v_h ) - \phi(\bar v_h)
			\\
			& \quad + L_\phi\|u_h - \bar u_h\|_U 
			+ (\nabla \widehat{J}(\bar u_h), \bar u_h - u_h)_{U} 
			\\
			& \leq \chi_{\text{gap},h}(u_h)
			+
			\|\nabla \widehat{J}(\bar u_h)-\nabla \widehat{j}_h(u_h)\|_U
			\mathrm{diam}(W_U)
			+ (\ell_{\nabla \widehat{J}} + L_\phi ) \rho_{\mathrm{proj}}(h).
		\end{align*}
		Moreover
		\begin{align*}
			\|\nabla \widehat{J}(\bar u_h)-\nabla \widehat{j}_h(u_h)\|_U
			\leq 
			\|\nabla \widehat{J}(u_h)-\nabla \widehat{j}_h(u_h)\|_U
			+ L_{\nabla \widehat{J}} \rho_{\mathrm{proj}}(h)
		\end{align*}
		and
		\begin{align*}
			\|\nabla \widehat{J}(u_h)-\nabla \widehat{j}_h(u_h)\|_U
			\leq 
			\|\Pi_h\nabla \widehat{J}(u_h)-\nabla \widehat{j}_h(u_h)\|_U
			+ 	\|\nabla \widehat{J}(u_h)-\Pi_h\nabla \widehat{J}(u_h)\|_U.
		\end{align*}
		Combined with \eqref{eq:Feb720241153}, and 
		\Cref{assumption:gelfand-projection,assumption:uniformbounded}, 
		we obtain the convergence bound.
	\end{proof}

	\Cref{lem:canonicalpsih} provides an example function
	satisfying $\phi(\Pi_h u) \leq \phi(u)$ for all $u \in U$
	and $h > 0$.
	In \cref{subsect:numillustrations},
	we empirically show that the upper bound provided by
	\Cref{thm:gapbased}	 cannot be improved in general.
	We use  a one dimensional boundary value problem
	(see \cref{subsec:linearproblem}) to demonstrate this.

	\section{Applications}
	
	Our technical assumption can be verified for PDE-constrained optimization,
	for example. 
	
	\subsection{PDE-constrained optimization}
	\label{subsec:pde-constrained-optimization}
	
	\Cref{assumption:trueproblem,assumption:discretizedproblem}
	have been motivated by PDE-constrained optimization.
	In this section, we demonstrate how to obtain operators $S$
	and $S_h$ satisfying the conditions in 
	\Cref{assumption:trueproblem,assumption:discretizedproblem}.
	In  PDE-constrained optimization, 
	the operator $S$ typically represents a PDE solution 
	or more generally the solution operator to 
	an operator equation.
	A large class of PDE solution operators satisfy the following framework
	inspired by those developed in \cite{Antil2018,Hinze2009,Ulbrich2011}.
	
	Let $Y$ and $U$ be as in \Cref{assumption:trueproblem}. 
	We consider the following conditions:
	\begin{enumerate}[nosep]
		\item $Z$ is a real Banach space, and
		$E : Y \times U \to Z$ is continuously differentiable,
		\item for each $u \in V_U$, $y=S(u) \in Y$
		is the unique solution to $E(y,u) = 0$, and
		\item $\Du_{y} E(S(u),u)$ has a bounded inverse for each
		$u \in V_U$.
	\end{enumerate}		
	Under these assumptions, $S \colon V_U \to Y$ is continuously
	differentiable according to the implicit function theorem.
	In the context of PDE-constrained optimization, 
	the approximations $S_h$ of $S$ typically satisfy 
	the following mathematical framework.
	See also \Cref{example:Eh}.
	
	Let $Y_h$ and $U_h$ be as in \Cref{assumption:discretizedproblem}. 
	For each $h > 0$, 
	\begin{enumerate}[nosep]
		\item $Z_h$ is a closed subspace of $Z$, and 			
		$E_h : Y_h \times U \to Z_h$ is continuously
		differentiable, 
		\item for each $u \in V_U$, 
		$y_h = S_h(u) \in Y_h$ is the unique solution to 
		$E_h(y_h, u) = 0$, and
		\item $\Du_{y_h} E_{h}(S_h(u),u)$ has a bounded inverse
		for each $u \in V_U$.
	\end{enumerate}
	Under these assumptions, $S_h \colon V_U \to Y$ is continuously
	differentiable for each $h > 0$.

	We  consider a typical set of operators
	$E_h$.
	
	\begin{example}
		\label{example:Eh}
		If $Z = Y^*$ and  $Z_h = Y_h^*$, then
		we can define
		\begin{align}
			\label{eq:Ehexample}
			E_h(y_h, u) \coloneqq  \iota_{Y_h,Y}^*E(\iota_{Y_h,Y}y_h, u).
		\end{align}
		However, choices other than
		\eqref{eq:Ehexample}
		are relevant for applications as is the case for stabilized
		finite element discretizations of convection-diffusion
		control problems \cite{Becker2007} or when quadrature
		and interpolation errors are considered \cite{Nochetto2006}, 
		for example. 
	\end{example}

	\subsection{Sparsity-promoting regularizers and box constraints}
	\label{subsect:sparsity}

	Sparsity-promoting regularizers continue to be of significant interest
	in PDE-constrained optimization and optimal control. 
	The section's main purpose is to verify
	\Cref{assumption:gelfand-projection,assumption:prox}
	for this contemporary $\psi$ and its canonical approximation $\psi_h$.
	
	Throughout the section, $\domain \subset \mathbb{R}^d$ 
	with $d \in \{1,2\}$ is a convex, 
	bounded domain. If $d = 2$, we require $\domain$ be a polygon.
	We refer the reader to \cite[Defs.\ 1.46 and 1.47]{Ern2004} for 
	definitions of domains and polygons.
	We consider $U \coloneqq L^2(\domain)$
	as the control space in \Cref{assumption:trueproblem}.
	Throughout the section, we consider the functions 
	$\psi : L^2(\domain) \to [0,\infty]$ and 
	$\psi_h : U_h \to [0,\infty]$ defined by
	\begin{align}
		\label{eq:psis}
		\begin{aligned}
			\psi(u) & \coloneqq  \beta \|u\|_{L^1(\domain)} + I_{[\lb,\ub]}(u),
			\quad \text{and} \quad 
			\psi_h(u_h) 
			\coloneqq 
			\beta \|u_h\|_{L^1(\domain)} + I_{U_h \cap [\lb_h, \ub_h]}(u_h),
		\end{aligned}
	\end{align}
	where $\lb_h$, $\ub_h \in U_h$, $\lb$, $\ub \in U$, and
	$\beta \geq 0$.

	To discretize $U$, we consider a family
	of successively refined meshes $(\mathcal{T}_h)_{h > 0}$ of $\domain$
	\cite[Defs.\ 1.49 and 1.53]{Ern2004}. 
	If $d = 1$, we require each $K \in \mathcal{T}_h$ be a closed interval
	and if $d = 2$, then  each $K \in \mathcal{T}_h$ is a
	(closed) triangle.
	For each  $K \in \mathcal{T}_h$, let $h_K$ be the Euclidean
	diameter of $K$ \cite[p.\ 32]{Ern2004}.
	We define $h \coloneqq \max_{K \in \mathcal{T}_h}\, h_K$.
	Let $\rho_K$ be the diameter of the largest Euclidean ball  in $K$.
	We require  $(\mathcal{T}_h)_{h > 0}$ be shape-regular
	and  quasi-uniform.
	Hence, there exist constants $\upsilon_{\mathcal{T}} > 0$
	and $\sigma_{\mathcal{T}} > 0$ such that
	\begin{align}
		\label{eq:shape-regular-quasi-uniform}
		h_K \leq \upsilon_{\mathcal{T}} \rho_K
		\quad \text{and} \quad 
		h_K \geq \sigma_{\mathcal{T}} h,
		\quad \text{for all} \quad K \in \mathcal{T}_h, \quad h > 0.
	\end{align}
	If $d = 1$, then $\upsilon_{\mathcal{T}} = 1$
	\cite[Rem.\ 1.108]{Ern2004}. If $d = 2$, then
	$\upsilon_{\mathcal{T}}$ can be computed using
	a formula provided in \cite[Rem.\ 1.108]{Ern2004}.
	Since each $K \in \mathcal{T}_h$
	is a triangle if $d = 2$
	and an interval if $d = 1$,
	each $K \in \mathcal{T}_h$ is convex. We denote by $|O|$ the 
	Lebesgue measure of a measurable subset $O \subset \mathbb{R}^d$
	and denote by $\accentset{\circ}{O}$ 
	the interior of a set $O \subset \mathbb{R}^d$.

	We choose $U_h$ as the space of piecewise constant functions. More precisely,
	we define
	\begin{align*}
		U_h \coloneqq 
		\{\,
		u_h \in L^\infty(\domain) \colon \, 
		u_{h|K}  \,\,\text{is constant for each} \,\,
		K \in \mathcal{T}_h \,\}.
	\end{align*}
	For each $v \in L^2(\domain)$, $h > 0$, 
	and $K \in \mathcal{T}_h$, we have
	(see, e.g., \cite[p.\ 73]{Ern2004})
	\begin{align*}
		[\Pi_h v](x) = 
		|\accentset{\circ}{K}|^{-1} (v,\mathbf{1})_{L^2(\accentset{\circ}{K})}
		\quad \text{for almost every} \quad x \in K.
	\end{align*}
	Moreover, there exists a constant $c_{\Pi} > 0$ such that
	\begin{align}
		\label{eq:projectionP0}
		\begin{aligned}
			\|\Pi_h v - v\|_{L^2(\accentset{\circ}{K})}
			&\leq c_{\Pi} h_K |v|_{H^1(\accentset{\circ}{K})}\quad &&\text{for all} \quad 
			v \in H^1(\accentset{\circ}{K}), \quad 
			K \in \mathcal{T}_h, \quad h > 0, 
			\\
			\|\Pi_h v - v\|_{L^2(\domain)}
			&\leq c_{\Pi} h |v|_{H^1(\domain)}\quad &&\text{for all} \quad 
			v \in H^1(\domain), \quad h > 0.
		\end{aligned}
	\end{align}
	Since $\domain \subset \mathbb{R}^d$
	and each $K \in \mathcal{T}_h$ are convex, we can choose
	$c_{\Pi} = 1/\pi$  in \eqref{eq:projectionP0} \cite[Thm.\ 3.2]{Bebendorf2003}.
	The first inequality \eqref{eq:projectionP0} can be established using
	Poincar\'{e}'s inequality, for example. The second inequality in
	\eqref{eq:projectionP0} can be deduced from the first one
	or from Proposition~1.135 in \cite{Ern2004}.

	The following lemma is used to verify
	\Cref{assumption:prox,assumption:gelfand-projection}.
	
	\begin{lemma}
		\label{lem:canonicalpsih}
		Let $\beta \geq 0$
		and let $\lb$, $\ub \in H^1(\domain) \cap L^\infty(\domain)$
		satisfy $\lb \leq \ub$.
		We define
		\begin{align*}
			\lb_h  \coloneqq \Pi_h \lb \quad \text{and} \quad \ub_h \coloneqq 
			\Pi_h \ub.
		\end{align*}
		Let $\psi$ and $\psi_h$ be defined as in \eqref{eq:psis}. 
		Then
		\begin{enumerate}[nosep]
			\item for all $h > 0$, $\tau > 0$, and  $v_h \in U_h$,
			\begin{align}
				\label{eq:projprox}
				\begin{aligned}
					\|\proj{[\lb,\ub]}{v_h}-\proj{U_h \cap [\lb_h,\ub_h]}{v_h}\|_{L^2(\domain)}
					& \leq c_{\Pi}h\big(|\lb|_{H^1(\domain)} + |\ub|_{H^1(\domain)}),
					\\
					\|\prox{\psi/\tau}{v_h}-\prox{\psi_h/\tau}{v_h}\|_{L^2(\domain)}
					& \leq c_{\Pi}h\big(|\lb|_{H^1(\domain)} + |\ub|_{H^1(\domain)}).
				\end{aligned}
			\end{align}
			\item $\psi$ is proper, convex,
			closed, and $\mathrm{dom}(\psi) = [\lb,\ub]$,
			\item  for all $h > 0$  and each $v \in [\lb,\ub]$,
			$\Pi_hv \in U_h \cap [\lb_h,\ub_h]$,
			\item for all $h > 0$, $\psi_h$ is proper, convex,
			and closed, and
			$\mathrm{dom}(\psi_h) = U_h \cap [\lb_h,\ub_h]$, 
			\item if $\lb \in \mathbb{R}$ and
			$\ub \in \mathbb{R}$, then
			$\mathrm{dom}(\psi_h) \subset \mathrm{dom}(\psi)$
			and $\psi_h = \psi$ on $U_h$ for all $h > 0$, and
			\item 
			for all $u \in L^2(\domain)$ and $h > 0$,
			$\|\Pi_h u\|_{L^1(\domain)} \leq \|u\|_{L^1(\domain)} $.
		\end{enumerate}
	\end{lemma}
	\begin{proof}
		\begin{enumerate}[nosep,wide]
			\item Fix $v_h \in U_h$. Since $U_h \subset L^2(\domain)$
			and $U_h$ is the space of piecewise constant functions, we have
			\begin{align*}
				\proj{[\lb,\ub]}{v_h} = \max\{\lb, \min\{v_h,\ub\}\},
				\quad \text{and} \quad 
				\proj{U_h \cap [\lb_h, \ub_h]}{v_h} = 
				\max\{\lb_h, \min\{v_h,\ub_h\}\};
			\end{align*}
			cf.\ \cite[p.\ 104]{Hinze2009}.
			For each $w \in \mathbb{R}$, the function
			$\mathbb{R}^2 \ni (t,s) \mapsto \max\{t,\min\{w,s\}\} \in \mathbb{R}$
			is Lipschitz continuous with respect to the Euclidean norm
			with Lipschitz constant one. Therefore, we obtain
			\begin{align*}
				\|\proj{[\lb,\ub]}{v_h}-\proj{U_h \cap[\lb_h,\ub_h]}{v_h}\|_{L^2(\domain)}^2
				\leq \|\lb - \Pi_h \lb\|_{L^2(\domain)}^2 +  
				\|\ub - \Pi_h\ub\|_{L^2(\domain)}^2.
			\end{align*}
			Combined with $(\varrho_1^2 + \varrho_2^2)^{1/2}
			\leq \varrho_1 + \varrho_2$ valid for all $\varrho_1$, $\varrho_2 \geq 0$
			and \eqref{eq:projectionP0}, we obtain the first assertion.

			Using \eqref{eq:psis}, we find that
			$\psi(u)/\tau  = I_{[\lb,\ub]}(u)/\tau +  (\beta/\tau) \|u\|_{L^1(\domain)}$
			for all $u \in L^2(\domain)$ and
			$\psi_h(u_h)/\tau  = I_{U_h\cap[\lb_h,\ub_h]}(u_h)/\tau + (\beta/\tau)  \|u_h\|_{L^1(\domain)}$
			for all $u_h \in U_h$.
			Combined with $U_h \subset L^2(\domain)$
			and the fact 
			that $U_h$ is the space of piecewise constant functions, we have
			for each $w_h \in U_h$,
			\begin{align*}
				\prox{\psi/\tau}{w_h}
				& = \proj{[\lb,\ub]}{w_h-\proj{[-\beta/\tau,\beta/\tau]}{w_h}},
				\\
				\prox{\psi_h/\tau}{w_h}
				&= \proj{U_h \cap [\lb_h,\ub_h]}{w_h-\proj{[-\beta/\tau,\beta/\tau]}{w_h}}
			\end{align*}
			cf.\ the computations in Lemma~C.1 in \cite{Milz2022c}.
			Defining $v_h \coloneqq w_h-\proj{[-\beta/\tau,\beta/\tau]}{w_h}$
			and observing $v_h \in U_h$, the first estimate  in 
			\eqref{eq:projprox} implies the second one.
			
			\item The set $[\lb,\ub]$ is nonempty, closed, and convex
			\cite[pp.\ 116--117]{Troeltzsch2010a}.
			Hence $I_{[\lb,\ub]}$ is proper, convex, and closed.
			H\"older's inequality and the boundedness of $\domain$ 
			ensure that $\|\cdot\|_{L^1(\domain)}$ is proper  on $L^2(\domain)$. 
			Combined with $[\lb,\ub] \subset L^2(\domain)$,
			$\mathrm{dom}(\psi) = [\lb,\ub]$.
			\item 
			The assertion follows from an application
			of Lemma~3 in \cite{Falk1973}.
			\item Using part~(c), we find that $\psi_h$ is proper.
			The remaining assertions can be established using arguments similar to
			those used in part~(b). 
			\item Using part~(e), we have
			$\mathrm{dom}(\psi_h) = U_h \cap [\lb,\ub]$.
			Combined with part~(b) and $U_h \subset L^2(\domain)$,
			we obtain $\mathrm{dom}(\psi_h) \subset \mathrm{dom}(\psi)$.
			\item This inequality follows from the projection's properties; cf.\ 
			\cite[Lem.\ 3.5]{Milz2022c}.
		\end{enumerate}
	\end{proof}

	\section{Numerical illustrations}
	
	We present numerical illustrations for 
	linear, semilinear, and bilinear
	PDE-constrained optimization.
	The problems are instances of
	\begin{align}
		\label{eq:semilinearcontrolproblem}
		\min_{u \in L^2(\domain)}\, 
		(1/2)\|S(u)-\widehat{y}\|_{L^2(\domain)}^2 
		+ \psi(u),
	\end{align}
	where $\psi$ is defined in \eqref{eq:psis},
	$\domain \coloneqq (0,1)^d$, $d \in \{1,2\}$, $\widehat{y} \in L^2(\domain)$, 
	$\beta \geq 0$, $\lb$, $\ub \in H^1(\domain) \cap L^\infty(\domain)$,	
	$V_{[\lb,\ub]}$ is an open neighborhood of $[\lb,\ub]$,
	and for each $u \in V_{[\lb,\ub]}$,
	$S(u) \in Y \coloneqq H_0^1(\domain)$ is the solution to either a 
	linear, semilinear, or bilinear PDE\@.
	We discretize the state space 
	$H_0^1(\domain)$ using 
	piecewise linear continuous finite element functions and 
	the control space $L^2(\domain)$ as described in 
	\eqref{subsect:sparsity} using the same meshes
	for both function space approximations.
	We discretize the operator equation defining the PDE solution 
	via \eqref{eq:Ehexample}.

	\subsection{Linear PDE-constrained optimization}
	\label{subsec:linearproblem}
	
	For our first test problem, 
	we consider a one dimensional computational domain,
	that is, $d=1$,  because it allows
	us perform simulations for relatively small mesh widths $h$ as compared
	to the problems with two dimensional domains
	considered in \cref{subsec:semilinearproblem,subsec:bilinearproblem}. 
	For $u \in L^2(\domain)$,
	we let $S(u) \in H_0^1(\domain)$ is the solution to the one dimensional
	boundary value problem
	\begin{align*}
		(\nabla y, \nabla v)_{L^2(\domain)} 
		= (u,v)_{L^2(\domain)}
		\quad \text{for all} \quad v \in H_0^1(\domain).
	\end{align*}
	Moreover, we define $\beta \coloneq 0.001$, $\lb(x) \coloneqq -1$,
	$\ub(x) \coloneqq 1 + (1/10)\sin(2\pi  x)$, and
	$\widehat{y}(x)  \coloneqq 100x^2$.
	
	\begin{figure}[t]
		\centering
		\subfloat{%
			\includegraphics[width=0.465\textwidth]
			{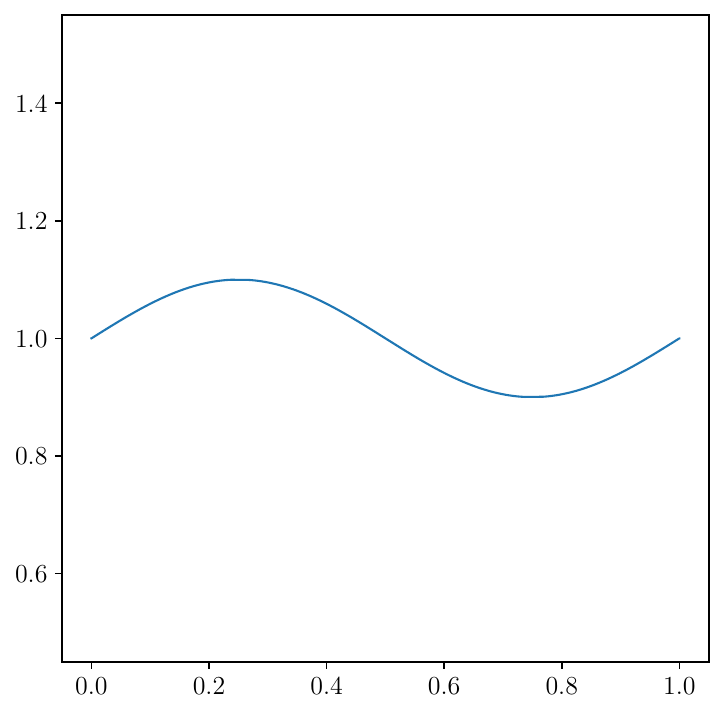}
		}
		\subfloat{%
			\includegraphics[width=0.465\textwidth]
			{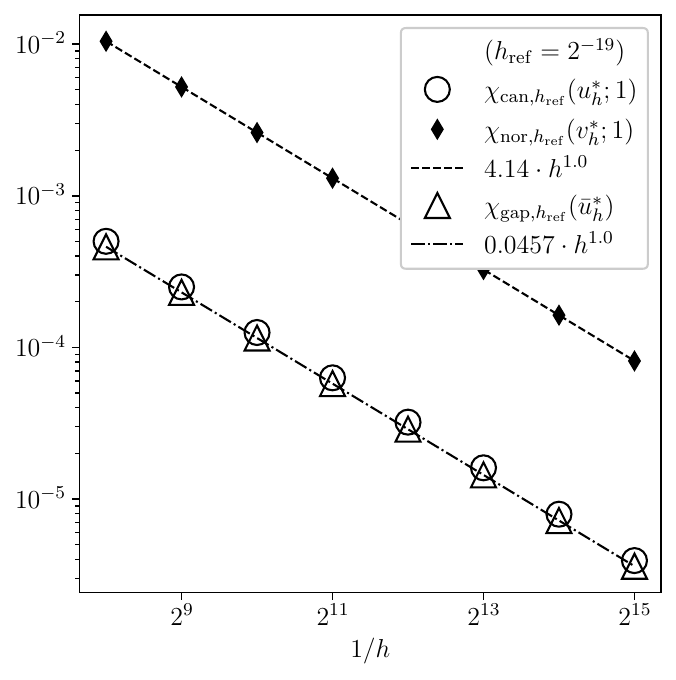}
		}
		\caption{
			Solution to the linear PDE-constrained optimization problem
			(see \cref{subsec:linearproblem})
			with mesh width parameter $h = 4096$
			\emph{(left)}, and 
			evaluations of the reference
			criticality measures 
			with reference discretization parameter
			$h_{\text{ref}} = 2^{-19}$
			\emph{(right)}.
			The controls $u_h^*$ are approximate critical points
			of the finite dimensional optimization problem corresponding
			to the discretization parameter $h> 0$,
			$v_h^*$ is computed according to \eqref{eq:vhcomputation},
			and $\bar u_h^*$ is computed according to \eqref{eq:baruhcomputation}.
			The convergence rates were computed using least squares.
		}
		\label{fig:boundaryvalueproblem}
	\end{figure}

	\subsection{Semilinear PDE-constrained optimization}
	\label{subsec:semilinearproblem}

	We choose $d = 2$.
	For semilinear PDE-constrained optimization, 
	for each $u \in L^2(\domain)$,
	$S(u) \in H_0^1(\domain)$ is the solution to the semilinear PDE
	\begin{align*}
		(\nabla y, \nabla v)_{L^2(\domain)^2} +
		(y^3, v)_{L^2(\domain)} = (u,v)_{L^2(\domain)}
		+ (g,v)_{L^2(\domain)}
		\quad \text{for all} \quad v \in H_0^1(\domain).
	\end{align*}

	We define $\lb(x) \coloneqq -10$.
	Moreover, $\ub(x) \coloneqq 0$
	if $(x_1, x_2) \in (0, 1/4) \times (0,1)$
	and
	$\ub(x) \coloneqq -5 + 20 x_1$ otherwise.
	Moreover, we 
	choose $\widehat{y}(x) \coloneqq 2\sin(4\pi x_1) \cos(8\pi x_2) \exp(2x_1)$, 
	$g(x) \coloneqq 10 \cos(8 \pi x_1)\cos(8 \pi x_2)$, and
	$\beta \coloneqq 0.0055$.
	Parts of this problem data is adapted from  \cite[Example 3]{Stadler2009}.
	
	We briefly comment on \Cref{assumption:uniformbounded}.
	Using  Assumption~(H1) and Lemma~3.8 in \cite{Casas2012}
	(see also Theorem~4.2 in \cite{Arada2002} 
	for computational domains with $C^{1,1}$ boundaries), 
	we can show that \eqref{eq:gradienterror} holds true
	with quadratic $\rho_{\nabla \widehat{J}}$.
	The Lipschitz condition in \Cref{assumption:uniformbounded}
	is implied by Theorem~10 in \cite{Hintermueller2004}.

	\subsection{Bilinear PDE-constrained optimization}
	\label{subsec:bilinearproblem}

	Deviating from \cref{subsec:semilinearproblem}, we define
	$\beta \coloneqq 0.0001$,
	$\widehat{y}(x) \coloneqq 1 + \sin(2\pi x_1) \sin(2\pi x_2)$, and
	$\lb(x) \coloneqq 0$, and consider a bilinear PDE instead of
	a semilinear PDE\@.
	According to \cite[sect.\ 7.1]{Milz2022d},
	there exists a neighborhood $V_{[\lb,\ub]}$ of $[\lb,\ub]$
	such that for each  $u \in V_{[\lb,\ub]}$,
	there exists a unique solution
	$y  = S(u)\in H_0^1(\domain)$  to the bilinear PDE
	\begin{align*}
		(\nabla y, \nabla v)_{L^2(\domain)^2} +
		(y u, v)_{L^2(\domain)} =  (g,v)_{L^2(\domain)}
		\quad \text{for all} \quad v \in H_0^1(\domain).
	\end{align*}
	
	We briefly comment on \Cref{assumption:uniformbounded}.
	Since $\|\Pi_ hv\|_{L^\infty(\domain)} \leq \|v\|_{L^\infty(\domain)}$
	for all $v \in L^\infty(\domain)$ and $h > 0$
	\cite[eq.\ (6.4)]{Hintermueller2004}, we can choose
	$W_U$ in \Cref{assumption:uniformbounded} to be a
	$L^\infty(\domain)$-bounded subset in $V_{[\lb,\ub]}$.
	For convex bounded domains $\domain \subset \mathbb{R}^2$
	with $C^{1,1}$ boundaries, Lemma~4.1 in \cite{Casas2018a} ensures
	\begin{align}
		\label{eq:bilineargradienterror}
		\|\nabla \widehat{J}_{h}(v)- \nabla\widehat{J}(v)\|_{L^\infty(\domain)}
		\leq C h,
	\end{align}
	valid for all $h > 0$ and $v \in W_U$, where $C = C(W_U) > 0$
	is a constant.  Owing to a discussion in 
	\cite[p.\ 4222]{Casas2018a}, and standard PDE regularity results, 
	we conjecture \eqref{eq:bilineargradienterror} holds true
	for $D = (0,1)^2$. In this case, \eqref{eq:bilineargradienterror}
	ensures  \eqref{eq:gradienterror}
	with linear  $\rho_{\nabla \widehat{J}}$.

	\subsection{Implementation details and numerical illustrations}
	\label{subsect:numillustrations}
	
	Our simulation environment  is built on the conditional gradient method
	\texttt{FW4PDE} \cite{Milz2024},
	\href{http://www.dolfin-adjoint.org/}{\texttt{dolfin-adjoint}} \cite{Mitusch2019,Funke2013,Farrell2013}, 
	\href{https://fenicsproject.org/}{\texttt{FEniCS}} \cite{Alnaes2015,Logg2012}, and 
	\href{https://github.com/funsim/moola}{\texttt{Moola}} 
	\cite{Nordaas2016,Schwedes2017}. 
	Given an  (approximate) critical point $u_h$
	of the finite dimensional
	problem, we evaluate the normal map-based criticality measure
	$\chi_{\text{nor}}(\cdot;\tau)$
	at 
	\begin{align}
		\label{eq:vhcomputation}
		v_h^* = u_h^*- (1/\tau) \nabla \widehat{j}_h(u_h^*);
	\end{align}
	cf.\ \Cref{lem:chinor}. We also compute
	\begin{align}
		\label{eq:baruhcomputation}
		\bar u_h^* \coloneqq \proj{\Pi_{h_{\text{ref}}}U_{\text{ad}}}{u_h^*};
	\end{align}
	cf.\ \Cref{thm:gapbased}.
	We approximate each criticality measures using the 
	finite dimensional problem's
	criticality measure with 
	$h = h_{\text{ref}}$.
	Specifically, we approximate $\chi_{\text{nor}}(\cdot;\tau)$
	via $\chi_{\text{nor},h_{\text{ref}}}(\cdot;\tau)$.
	Similar approximations are used for the canonical criticality measure
	and the gap function.
	
	Our computer code is available
	at \url{https://github.com/milzj/ErrorEstimation}.
	\Cref{fig:boundaryvalueproblem} depicts a solution
	to the linear PDE-constrained optimization problem and
	the convergence rates of  the three criticality measures
	with $\tau = 1$.
	For each criticality measure, the convergence rates are 
	close to $1$, empirically illustrating our theoretical 
	convergence rate analysis.
	\Cref{fig:solutions} depicts critical points to the semilinear and 
	bilinear PDE-constrained optimization problems.
	The convergence rates 
	for the normal map and the canonical criticality measure
	depicted in \Cref{fig:convergence-rates}
	illustrate our theoretical convergence statements. 
	The gap function-based convergence rates depicted in 
	\Cref{fig:convergence-rates} are higher than the upper
	bounds established in \cref{subsection:gapfunction}
	and those depicted in \Cref{fig:boundaryvalueproblem}.
	While in no contradiction with our theory, we believe that
	this disqualifies the gap function-based convergence rates
	to be used reliably for code verifications and validations.

	\begin{figure}[t]
		\centering
		\subfloat{%
			\includegraphics[width=0.465\textwidth]
			{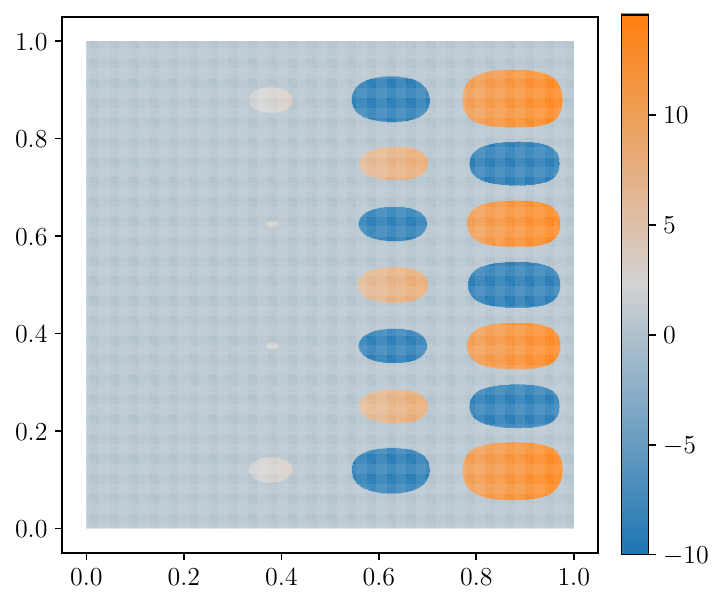}}
		\subfloat{%
			\includegraphics[width=0.465\textwidth]
			{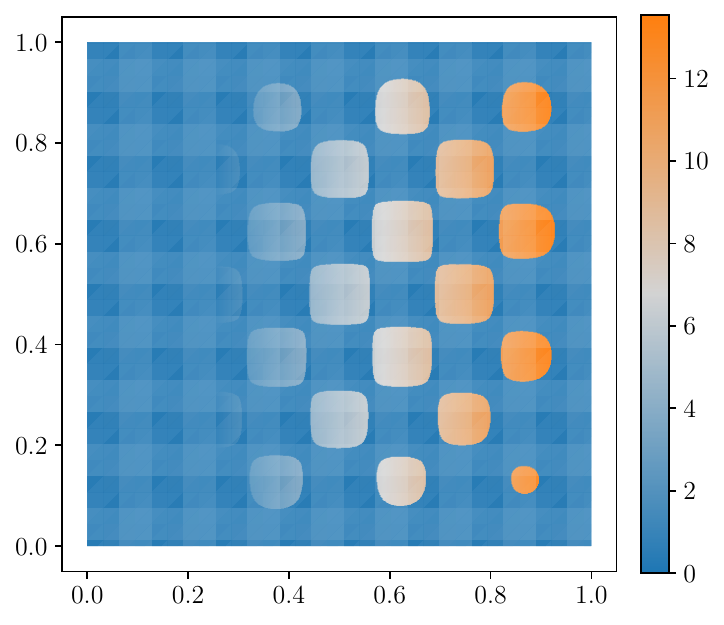}}
		\caption{
			Critical point of the semilinear PDE-constrained
			optimization problem 
			(see \cref{subsec:semilinearproblem})
			\emph{(left)} and 
			critical point of the bilinear PDE-constrained optimization problem
			(see \cref{subsec:bilinearproblem})
			\emph{(right)} with mesh discretization parameter
			$h = \sqrt{2}/512$.
		}
		\label{fig:solutions}
	\end{figure}
	
	\begin{figure}[t]
		\centering
		\subfloat{%
			\includegraphics[width=0.465\textwidth]
			{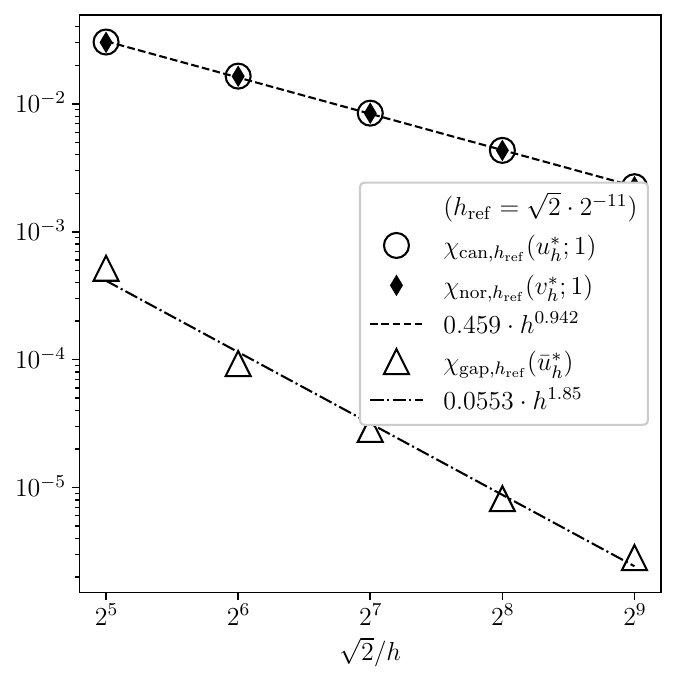}}
		\subfloat{%
			\includegraphics[width=0.465\textwidth]
			{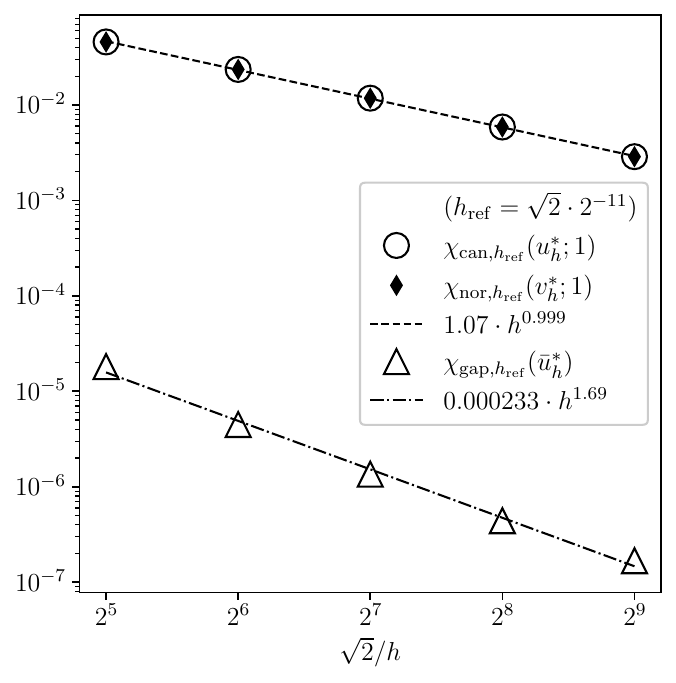}}
		\caption{%
			Evaluations of the reference criticality measures
			for the semilinear PDE-constrained optimization problem
			\emph{(left)} and the blinear PDE-constrained optimization problem
			\emph{(right)}
			with reference discretization parameter
			$h_{\text{ref}} = \sqrt{2}\cdot 2^{-11}$.
			The controls $u_h^*$ are approximate critical points
			of the finite dimensional optimization problem corresponding
			to the discretization parameter $h> 0$,
			$v_h^*$ is computed according to \eqref{eq:vhcomputation},
			and $\bar u_h^*$ is computed according to \eqref{eq:baruhcomputation}.
			The convergence rates were computed using least squares.
		}
		\label{fig:convergence-rates}
	\end{figure}

	\section{Discussion}
	
	Proposing effective approximation and discretization schemes
	and analyzing their approximation power is a fundamental
	task of utmost importance to tackle real-world decision making 
	and estimation tasks. This framework is needed in many research
	fields, such as
	science and engineering, computational optimization, and statistics.
	The ``analysis part'' essentially requires 
	two ingredients: (i)  a meaningful way of how to quantify
	accuracy of a discretization scheme and (ii) the mathematical exercise
	of deriving (order-optimal) error estimates.
	
	We proposed quantifying discretization accuracy using criteria used
	to terminate numerical solution schemes. Specifically, we focused on
	criticality measures used within termination criteria for 
	numerical methods, such as (conditional) gradient methods, 
	and semismooth Newton methods.
	We have demonstrated the optimality of our error estimates on  simple
	optimization problems. Our theoretical analysis is
	empirically validated on two nonconvex PDE-constrained optimization
	problems.

	Our error estimates are order-optimal within the problem class considered in this
	manuscript. For more restrictive problem classes we have observed
	faster convergence rates  empirically. 
	Specifically, we observed improved
	convergence rates for the canonical criticality
	measure and the unregularized gap function for optimization problems
	with \emph{simple composite functions}
	where $\psi_h \coloneqq \psi$ on $U_h$ for each $h >0$.
	This includes feasible sets given by box constraints with constant bounds
	\cite{Meyer2004}, norm balls,  
	and volume constraints \cite{Haubner2023}, for example.
	We will provide a rigorous analysis in further work.
	Our criticality measure-based error estimates will allow us to derive
	distance-based error estimates through the notion
	\emph{strong metric subregularity}
	\cite{Dontchev2009}. This will also be covered in future work.

	\begin{footnotesize}
		\bibliography{ErrorEstimates4PDE-arXiv}
	\end{footnotesize}
\end{document}